\documentclass[11pt,a4paper]{amsart}
\usepackage{url}
\usepackage{amsmath,amsthm,amsfonts,amssymb,latexsym}
\usepackage{graphicx}
\usepackage{rotating}
\usepackage[shortlabels]{enumitem}
\usepackage[T1]{fontenc}

\date{\today}

\newtheorem{theorem}{Theorem}[section]

\newtheorem{lemma}[theorem]{Lemma}

\newtheorem{corollary}[theorem]{Corollary}
\newtheorem{proposition}[theorem]{Proposition}

\newtheorem{question}[theorem]{Question}

\theoremstyle{definition}
 \newtheorem{definition}[theorem]{Definition}
 \newtheorem{remark}[theorem]{Remark}
 
\newcommand{\IR}{\mathbb R}
\newcommand{\IC}{\mathbb C}

\newcommand{\w}{\omega}

\newcommand{\IP}{\mathbb P}

\newcommand{\C}{\mathcal{C}}

\newcommand{\HH}{\mathcal{H}}

\newcommand{\A}{\mathcal{A}}

\newcommand{\cov}{\operatorname{cov}}
\newcommand{\cof}{\operatorname{cof}}

\newcommand{\cc}{\mathfrak c}

\newcommand{\vid}{\hat{\ }}
\newcommand{\bigvid}{\hat{\ \ }}

\newcommand{\uhr}{\upharpoonright}
\newcommand{\cf}{\mathrm{cf}}

\newcommand{\name}[1]{\dot{#1}}
\newcommand{\la}{\langle}
\newcommand{\ra}{\rangle}

\newcommand{\forces}{\Vdash}

\newcommand{\nothing}[1]{}

\newcommand{\zero}{\normalfont{\textbf{0}}}
\newcommand{\one}{\normalfont{\textbf{1}}}

\title[Convergence of measures in forcing extensions]{Convergence of measures in forcing extensions}

\author{Damian Sobota  and Lyubomyr Zdomskyy}

\address{Institut f\"ur Diskrete Mathematik und Geometrie, Technische Universit\"at Wien, Wiedner Hauptstra\ss e 8-10/104, 1040 Wien, Austria.}
\email{ein.damian.sobota@gmail.com}
\email{lzdomsky@gmail.com}


\thanks{ The  authors would
like to thank  the Austrian Science Fund FWF (Grant I 2374-N35)
 for generous support for this research.}

\begin{document}
\maketitle

\begin{abstract}
We prove that if $\A$ is a $\sigma$-complete Boolean algebra in a model $V$ of set theory and $\IP\in V$ is a proper forcing with the Laver property preserving the ground model reals non-meager, then every pointwise convergent sequence of measures on $\A$ in a $\IP$-generic extension $V[G]$ is weakly convergent, i.e. $\A$ has the Vitali--Hahn--Saks property in $V[G]$. This yields a consistent example of a whole class of infinite Boolean algebras with this property and of cardinality strictly smaller than the dominating number $\mathfrak{d}$. We also obtain a new consistent situation in which there exists an Efimov space.  
\end{abstract}

\section{Introduction}

It is a standard topic in Banach space theory to investigate whether the convergence of functionals on a given Banach space in one topology implies the convergence in another finer one. In this paper we study the following instance of this problem. Let $\A$ be a $\sigma$-complete Boolean algebra. It follows from Nikodym's Uniform Boundedness Principle (see Diestel \cite[page 80]{Die84}) that every pointwise convergent sequence of measures on $\A$ is also weak* convergent (i.e. $\A$ has \textit{the Nikodym property}; see Section \ref{section:convergence} for definitions). On the other hand, Grothendieck \cite{Gro53} proved that every weak* convergent sequence of measures on $\A$ is weakly convergent (i.e. $\A$ has \textit{the Grothendieck property}). Thus, it follows that if $K_\A$ is the Stone space of a $\sigma$-complete Boolean algebra $\A$, then the pointwise convergence of a sequence in the dual space $C\big(K_\A\big)^*$ implies its weak convergence (i.e. $\A$ has the \textit{Vitali--Hahn--Saks property}).

Intuitively speaking, the Stone spaces of Boolean algebras with the Nikodym property do not admit any non-trivial pointwise convergent sequences of measures --- here, ``non-trivial'' means ``not weak* convergent''. Similarly, the Stone spaces of Boolean algebras with the Grothendieck property do not have any non-trivial weak* convergent sequences of measures, where ``non-trivial'' this time means ``not weakly convergent''. It is well-known that the Stone spaces of Boolean algebras with any of the Nikodym or Grothendieck properties do not have any non-trivial, i.e. non-eventually constant, convergent sequences, so both properties may be seen as strengthening of the lack of non-trivial convergent sequences in the Stone spaces. As we see further, this is related to the famous Efimov problem. 

Let now $V$ denote the set-theoretic universe, $\IP\in V$ be a notion of forcing and $G$ a $\IP$-generic filter over $V$. Assume that $\A\in V$ is a $\sigma$-complete Boolean algebra.  Preservation of the Vitali--Hahn--Saks property in the extension $V[G]$ is not automatic --- e.g. if $\IP$ adds new reals, then the ground model algebra $\big(\wp(\w)\big)^V$ of all subsets of integers will no longer be $\sigma$-complete in $V[G]$ and it may also fail to have the Vitali--Hahn--Saks property --- such a situation happens e.g. after adding Cohen reals (see Dow and Fremlin \cite[Introduction]{DF07}). The main aim of this paper is thus to find out what properties of $\IP$ are sufficient to ensure that $\A$ will have the Vitali--Hahn--Saks property in the extension $V[G]$.

Our question was motivated by the utility of the properties --- the research of Seever \cite{See68}, Talagrand \cite{Tal80}, Haydon \cite{Hay81}, Molt\'o \cite{Mol81}, Schachermayer \cite{Sch82}, Freniche \cite{Fre84}, Aizpuru \cite{Aiz92}, Valdivia \cite{Val13}, K\k{a}kol and Lop\'ez-Pellicer \cite{KakLopPel16} etc. showed their importance. Also the following cardinal number issue led us to deal with the problem. It can be shown that many separation or interpolation properties of infinite Boolean algebras studied by the aforementioned authors implying the Nikodym or Grothendieck properties imply also that these algebras have cardinality at least equal to the continuum $\cc$. The natural question whether consistently there are infinite Boolean algebras with at least one of the properties and of cardinality strictly less than $\cc$ appeared. Brech \cite{Bre06} showed that in the side-by-side Sacks extension all ground model $\sigma$-complete Boolean algebras preserve the Grothendieck property. Recently, Sobota and Zdomskyy \cite{SobZdo17} proved the same result for the Nikodym property, which --- together with Brech's result --- consistently yields a class of examples of Boolean algebras with the Vitali--Hahn--Saks property and of cardinality $\omega_1$ while the inequality $\omega_1<\cc$ holds true in the model. Also, Sobota \cite{Sob17} proved in ZFC that for every cardinal number $\kappa$ such that $\cof([\kappa]^\w)=\kappa\ge\cof(\mathcal{N})$ (the cofinality of the Lebesgue null ideal) there exists a Boolean algebra with the Nikodym property and of cardinality $\kappa$.

In this paper we generalize the results of Brech \cite{Bre06} and the authors \cite{SobZdo17}. Namely, we prove (Theorem \ref{theorem:main}) that if $\IP$ is a proper forcing preserving the ground model reals as a non-meager subset of the reals in the extension and having the Laver property (Definitions \ref{def:preservation_nonmeager} and \ref{def:laver_property}), then in any $\IP$-generic extension any ground model $\sigma$-complete Boolean algebra has the Vitali--Hahn--Saks property. There are many examples of notions of forcing satisfying the assumptions of the theorem, e.g. side-by-side products of Sacks forcing, Miller forcing or Silver(-like) forcing.


Our result has some interesting consequences. First, it yields a consistent example of a whole class of infinite Boolean algebras with the Vitali--Hahn--Saks property and of cardinality strictly less than the dominating number $\mathfrak{d}$, as well as it sheds some new light on connections between convergence of measures on Boolean algebras and cardinal characteristics of the continuum --- see Section \ref{section:cardinal}. Second, as shown in Section \ref{section:efimov}, it can be used to obtain a new consistent situation in which there exists \textit{an Efimov space} --- a counterexample to \textit{the Efimov problem}, a long-standing open question asking whether there exists an infinite compact Hausdorff space containing neither any non-trivial convergent sequence nor any copy of $\beta\omega$, the \v{C}ech-Stone compactification of integers. So far, Efimov spaces have been obtained only with an aid of additional set-theoretic assumptions (like $\Diamond$) --- no ZFC Efimov space is known; see Hart \cite{Har07} for a detailed discussion.

The structure of the paper is as follows. In Section \ref{section:convergence} we recall basic definitions, properties and facts concerning (sequences of) measures on Boolean algebras. In Section \ref{section:trees_on_forcings} we construct special auxiliary trees on a poset $\IP$ associated with the Nikodym and Grothendieck properties. Section \ref{section:aux_set_theory} contains auxiliary results concerning almost disjoint families and proper posets. In Section \ref{section:main} we present the proof of the main result --- Theorem \ref{theorem:main}. Section \ref{section:consequences} presents the consequences of the theorem to cardinal characteristics of the continuum and Efimov spaces.


\section{Measures on Boolean algebras\label{section:convergence}}

In this section we provide notation, terminology, basic definitions and facts concerning sequences of measures used in the paper.

Let $\A$ be a Boolean algebra. The Stone space of $\A$ is denoted by $K_\A$. Recall that by the Stone duality $\A$ is isomorphic to the algebra of clopen subsets of $K_\A$; if $A\in\A$, then $[A]$ denotes the corresponding clopen subset of $K_\A$. A subset $X$ of $\A$ is an \textit{antichain} if $x\wedge y=\zero_\A$ for every distinct $x,y\in X$, i.e. every two distinct elements of $X$ are \textit{disjoint}.

A \textit{measure} $\mu\colon\A\to\IC$ on $\A$ is always a finitely additive complex-valued function with finite variation $\|\mu\|$. The measure $\mu$ has a unique regular Borel extension (denoted also by $\mu$) onto the space $K_\A$, preserving the variation of $\mu$. By the Riesz representation theorem the dual space $C\big(K_\A\big)^*$ of the Banach space $C\big(K_\A\big)$ of continuous complex-valued functions on $K_\A$, endowed with the supremum norm, is isometrically isomorphic to the space of all regular Borel measures on $K_\A$, hence $C\big(K_\A\big)^*$ is isometrically isomorphic to the space of all measures on $\A$. 

Let us now recall basic definitions concerning sequences of measures. 

\begin{definition}
Let $\A$ be a Boolean algebra. We say that a sequence $\la\mu_n\colon\ n\in\w\ra$ of measures on $\A$ is:
\begin{itemize}
	\item \textit{pointwise bounded} if $\sup_{n\in\w}\big|\mu_n(A)\big|<\infty$ for every $A\in\A$;
	\item \textit{uniformly bounded} if $\sup_{n\in\w}\big\|\mu_n\big\|<\infty$;
	\item \textit{pointwise convergent} if $\lim_{n\to\infty}\mu_n(A)$ exists for every $A\in\A$;
	\item \textit{weak* convergent} if $\lim_{n\to\infty}\int_{K_\A}fd\mu_n$ exists for every $f\in C\big(K_\A\big)$;
	\item \textit{weakly convergent} if $\lim_{n\to\infty}x^{**}\big(\mu_n\big)$ exists for every $x^{**}\in C\big(K_\A\big)^{**}$.
\end{itemize}
\end{definition}

\begin{remark}\label{remark:weak_conver_borel}
Note that the weak convergence is equivalent to the convergence on every Borel subset of $K_\A$ --- see Diestel \cite[Theorem 11, page 90]{Die84}.
\end{remark}

\begin{definition}
We say that a Boolean algebra $\A$ has:
\begin{itemize}
	\item \textit{the Nikodym property} if every pointwise bounded sequence of measures on $\A$ is uniformly bounded;
	\item \textit{the Grothendieck property} if every weak* convergent sequence of measures on $\A$ is weakly convergent;
	\item \textit{the Vitali--Hahn--Saks property} if every pointwise convergent sequence of measures on $\A$ is weakly convergent.
\end{itemize}
\end{definition}

Proposition \ref{prop:nikodym_convergence}, which is a simple folklore fact, shows that the definition of the Nikodym property can be also stated in the convergence manner: a Boolean algebra $\A$ has the Nikodym property if every pointwise convergent sequence of measures on $\A$ is weak* convergent. However, the first definition is easier to deal with (and also follows from the original statement of Nikodym's theorem \cite{Nik33}). The Vitali--Hahn--Saks property is usually stated in terms of so-called exhaustiveness of families of measures, but Schachermayer \cite[Theorem 2.5]{Sch82} proved that the property is equivalent to the conjunction of the Nikodym and Grothendieck properties, whence follows our definition. Note that by Remark \ref{remark:weak_conver_borel} the definition of the property can be stated also as follows: a Boolean algebra $\A$ has the Vitali--Hahn--Saks property if every pointwise convergent sequence of measures on $\A$ is convergent on every Borel subset of $K_\A$.

\begin{proposition}\label{prop:nikodym_convergence}
Let $\A$ be a Boolean algebra. Then, the following are equivalent:
\begin{enumerate}
	\item every pointwise convergent sequence of measures on $\A$ is weak* convergent;
	\item every pointwise convergent sequence of measures on $\A$ is uniformly bounded;
	\item every pointwise bounded sequence of measures on $\A$ is uniformly bounded.
\end{enumerate}
\end{proposition}
\begin{proof}
\noindent(1)$\Rightarrow$(2): By the Banach--Steinhaus theorem (the Uniform Boundedness Principle) every weak* convergent sequence of measures on $\A$ is uniformly bounded.

\medskip

\noindent(2)$\Rightarrow$(3): Assume there is a pointwise bounded sequence $\la\mu_n\colon\ n\in\w\ra$ of measures on $\A$ such that $\lim_{n\to\infty}\big\|\mu_n\big\|=\infty$. For every $n\in\w$ define the measure:
\[\nu_n=\mu_n/\sqrt{\big\|\mu_n\big\|}.\]
Then, $\la\nu_n\colon\ n\in\w\ra$ is pointwise convergent to $0$. Indeed, for every $A\in\A$ and $n\in\w$ we have:
\[\big|\nu_n(A)\big|=\frac{\big|\mu_n(A)\big|}{\sqrt{\big\|\mu_n\big\|}}\le\frac{\sup_{m\in\w}\big|\mu_m(A)\big|}{\sqrt{\big\|\mu_n\big\|}},\]
so $\lim_{n\to\infty}\nu_n(A)=0$ for every $A\in\A$. On the other hand, since $\big\|\nu_n\big\|=\sqrt{\big\|\mu_n\big\|}$ for every $n\in\w$, we have $\lim_{n\to\infty}\big\|\nu_n\big\|=\infty$, a contradiction with (2).

\medskip 

\noindent(3)$\Rightarrow$(1): Let $\la\mu_n\colon\ n\in\w\ra$ be a sequence of measures on $\A$ pointwise convergent to $0$. Fix $f\in C\big(K_\A\big)$ and let $\varepsilon>0$. Since $\la\mu_n\colon\ n\in\w\ra$ is pointwise bounded, it is by (3) uniformly bounded. Let then $M>0$ be such that $\big\|\mu_n\big\|<M$ for every $n\in\w$. Let $\sum_{i=1}^k\alpha_i\chi_{A_i}\in C\big(K_\A\big)$ ($\alpha_i\in\IC,A_i\in\A$) be such a function that:
\[\big\|f-\sum_{i=1}^k\alpha_i\chi_{A_i}\big\|<\varepsilon/(2M).\]
By the pointwise convergence to $0$, there is $N\in\w$ such that for every $n>N$ we have:
\[\sum_{i=1}^k\big|\alpha_i\big|\cdot\big|\mu_n\big(A_i\big)\big|<\varepsilon/2.\]
Then, for every $n>N$ it holds:
\[\Big|\int_{K_\A}fd\mu_n\Big|\le\int_{K_\A}\big|f-\sum_{i=1}^k\alpha_i\chi_{A_i}\big|d\mu_n+\int_{K_\A}\Big|\sum_{i=1}^k\alpha_i\chi_{A_i}\Big|d\mu_n\le\]
\[\big\|f-\sum_{i=1}^k\alpha_i\chi_{A_i}\big\|\cdot\big\|\mu_n\big\|+\sum_{i=1}^k\big|\alpha_i\big|\cdot\big|\mu_n\big(A_i\big)\big|<\big(\varepsilon/(2M)\big)\cdot M+\varepsilon/2=\varepsilon,\]
which proves (3).
\end{proof}


\subsection{Anti-Nikodym sequences}

\begin{definition}
A sequence $\la\mu_n\colon\ n\in\w\ra$ of measures on a Boolean algebra $\A$ is \textit{anti-Nikodym} if it is pointwise bounded on $\A$ but not uniformly bounded.
\end{definition}

\noindent Obviously, a Boolean algebra has the Nikodym property if it does not have any anti-Nikodym sequences of measures. The following lemma is an important tool in studying the Nikodym property --- see Sobota \cite[Lemmas 4.4 and 4.7]{Sob17} for a proof.

\begin{lemma}\label{lemma:aN_antichain}
Let $\A$ be a Boolean algebra and $\la\mu_n\colon\ n\in\w\ra$ an anti-Nikodym sequence of measures on $\A$. Then, there exists $x\in K_\A$ such that 
for every finite $\A_0\subset \A$ with $x\not\in\bigvee \A_0$ and $M>0$, there exist $X\in[\w]^\w$ and an antichain $\big\{A_n\colon\ n\in X\big\}$ in $\one_\A\setminus\bigvee\A_0$ such that for every $n\in X$ we have $x\not\in A_n$ and the following inequality holds:
\[\big|\mu_n(A_n)\big|>\max_{\C\subset\A_0}\big|\mu_n\big(\bigvee\C\big)\big|+M.\]\hfill$\Box$
\end{lemma}

Every point $x\in K_\A$ satisfying the conclusion of Lemma \ref{lemma:aN_antichain} is called \textit{a Nikodym concentration point of the sequence $\la\mu_n\colon\ n\in\w\ra$}. Properties of the set of all Nikodym concentration points of a given sequence were studied in Sobota \cite[Section 4]{Sob17}.

\subsection{Anti-Grothendieck sequences}

Similarly to anti-Nikodym sequences we define anti-Grothendieck sequences.

\begin{definition}
A sequence $\la\mu_n\colon\ n\in\w\ra$ of measures on a Boolean algebra $\A$ is \textit{anti-Grothendieck} if it is weak* convergent to the zero measure but not weakly convergent.
\end{definition}

Note that by the Banach--Steinhaus theorem (the Uniform Boundedness Principle) every anti-Grothendieck sequence is uniformly bounded.

A (far) analogon of Lemma \ref{lemma:aN_antichain} for anti-Grothendieck sequences is the following well-known consequence of the Dieudonn\'e--Grothendieck theorem (see e.g. Diestel \cite[Theorem VII.14, page 98]{Die84}).

\begin{lemma}\label{lemma:aG_antichain}
Let $\A$ be a Boolean algebra and $\la\mu_n\colon\ n\in\w\ra$ an anti-Grothendieck sequence of measures on $\A$. Then, there exist $X\in[\w]^\w$, an antichain $\la A_n\colon\ n\in X\ra$ in $\A$, and $\varepsilon>0$ such that for every $n\in X$ the following inequality holds:
\[\big|\mu_n\big(A_n\big)\big|>\varepsilon.\]\hfill$\Box$
\end{lemma}

\subsection{Measures and almost disjoint families}

The following fact is folklore --- see e.g. Brech \cite[Lemma 2.1]{Bre06} and Sobota \cite[Lemma 2.6]{Sob17} for different proofs. Recall that a family $\HH\subset[\w]^\w$ is \textit{almost disjoint} if $A\cap B$ is finite for every distinct $A,B\in\HH$.

\begin{lemma}\label{lemma:measures_ad}
Let $\HH$ be an uncountable family of infinite almost disjoint subsets of $\omega$ and let $\la A_n\colon\ n\in\w\ra$ be an antichain in a Boolean algebra $\A$. Assume that $\bigvee_{n\in H}A_n\in\A$ for every $H\in\HH$. Then, for every sequence $\la\mu_n\colon\ n\in\w\ra$ of measures on $\A$ there exists $H_0\in\HH$ such that for every $k\in\w$ the following equality holds:
\[\mu_k\Big(\bigvee_{n\in H_0}A_n\Big)=\sum_{n\in H_0}\mu_k\big(A_n\big).\]\hfill$\Box$
\end{lemma}

\section{Trees on forcings and sequences of measures\label{section:trees_on_forcings}}

Throughout the \underline{whole} paper $V$ denotes the set-theoretic universe. In the following section we assume that $\IP$ is a notion of forcing and $G$ is a $\IP$-generic filter over the ground model $V$.

\begin{proposition}\label{prop:tree_nikodym}
Let $\A$ be a ground model Boolean algebra. Let $\la\name{\mu}_n\colon\ n\in\w \ra$ be a sequence of $\IP$-names for measures on $\A$ and $\name{x}$ a name for a point in $K_\A$. Assume that $1_{\IP}$ forces the following formulas:
\[``\la\name{\mu}_n\colon\ n\in\w \ra\text{ is anti-Nikodym}",\]
\[``\name{x}\text{ is a Nikodym concentration point of }\la\name{\mu}_n\colon\ n\in\w \ra",\text{ and}\]
\[``\big\|\name{\mu}_n\big\|<n\text{ for every }n\in\w".\]
Then, there exists a tree $T\subset\IP^{<\omega}$ such that to every $t=\la p_0,\ldots,p_{n-1}\ra\in T$ ($n\ge1$) there are associated the following objects:
\begin{itemize}
	\item a set $a_t\in[\omega]^{<\omega}$,
	\item a sequence $\la A_m^t\in\A\colon\ m\in a_t\ra$,
	\item a sequence $\la b_{t\uhr k}^t\subset a_{t\uhr k}\colon\ 1\le k<n\ra$,
	\item a sequence $\la e_{t\uhr k,m}^t\subset b_{t\uhr k}^t\colon\ m\in a_{t\uhr i},1\le i<k<n\ra$,
	\item a sequence $\la l_{t\uhr k}^t\in b_{t\uhr k}^t\colon\ 1\le k<n\ra$, 
\end{itemize}
satisfying the following conditions:
\begin{enumerate}[(i)]
	\item $\max a_{t\uhr k}<\min a_{t\uhr (k+1)}$ for all $1\le k<n$,
	\item $\big|a_{t\uhr k}\big|>(k+1)\cdot\big|b_{t\uhr k}^t\big|$ for all $1\le k<n$,
	\item $b_{t\uhr k}^t=\big\{l_{t\uhr k}^t\big\}\cup\bigcup_{1\le i<k}\bigcup_{m\in a_{t\uhr i}}e_{t\uhr k,m}^t$ for all $1\le k<n$,
	\item $\big\{A^{t\uhr k}_m\colon\ m\in a_{t\uhr k},m\neq l_{t\uhr k}^t,1\le k<n\big\}\cup\big\{A^t_m\colon\ m\in a_t\big\}$ is an antichain in $\A$,
	\item $p_{n-1}\forces e_{t\uhr k,m}^t=\big\{l\in a_{t\uhr k}\colon\ \big|\name{\mu}_m\big|\big(A_l^{t\uhr k}\big)\ge 1/2^k\big\}$ for all $m\in a_{t\uhr i}$ and $1\le i<k<n$,
	\item either $p_{n-1}\forces\Big(l_{t\uhr k}^t=\min a_{t\uhr k}\text{ and }\name{x}\not\in\bigvee_{m\in a_{t\uhr k}}A_m^{t\uhr k}\Big)$, or $p_{n-1}\forces\name{x}\in A_{l_{t\uhr k}^t}^{t\uhr k}$, for all $1\le k<n$,
	\item $p_{n-1}\forces\forall \C\subset\big\{A_l^{t\uhr k}\colon\ l\in a_{t\uhr k},l\neq l_{t\uhr k}^t,1\le k<n\big\}\ \forall m\in a_t\colon\ $
\[\big|\name{\mu}_m\big(A_m^t\big)\big|>\big|\name{\mu}_m(\bigvee\C)\big|+n.\]
\end{enumerate}
Besides, for every $t\in T$ the set $D_t^T=\big\{q\colon\ t\bigvid q\in T\big\}\in V$ is dense in $\IP$.
\end{proposition}
\begin{proof}
%

The first level of the tree $T$ is constructed as follows. Fix $p\in\IP$. Since $1_\IP$ forces that $\name{x}$ is a Nikodym concentration point of $\la\name{\mu}_n\colon\ n\in\w \ra$, by Lemma \ref{lemma:aN_antichain}, there exist $q_p\le p$, a name $\name{X}$ for an infinite subset of $\omega$ and a name $\name{f}$ for a map from $\name{X}$ to $\A$ whose range is an antichain with elements not containing $\name{x}$, such that $q_p$ forces that for every $n\in\name{X}$ we have 
\[\big|\name{\mu}_n\big(\name{f}(n)\big)\big|>1.\]

Now, there exist $r_p\le q_p$, $a_{r_p}\in[\w]^{<\w}$ of size $\big|a_{r_p}\big|>2^3$ and $\la A_m^{r_p}\in\A\colon\ m\in a_{r_p}\ra$ such that $r_p$ forces that $a_{r_p}\subset\name{X}$ and $\name{f}(m)=A_m^{r_p}$ for every $m\in a_{r_p}$ (hence $\la A_m^{r_p}\colon\ m\in a_{r_p}\ra$ is also an antichain). 

The map $p\mapsto r_p$ (with the domain $\IP$) has the range which is a dense subset of $\IP$ --- call this range $D_\emptyset^T$, i.e. $D_\emptyset^T$ will be the first level of the tree $T$. Note that trivially $D_\emptyset^T\in V$. For every $r_p\in D_\emptyset^T$ and $m\in a_{r_p}$ rename $a_{\la r_p\ra}=a_{r_p}$ and $A_m^{\la r_p\ra}=A_m^{r_p}$ (if for different $p$ and $p'$ we have got $r_p=r_{p'}$, then assume that $a_{r_p}=a_{r_{p'}}$ and $A_m^{r_p}=A_m^{r_{p'}}$ for all $m\in a_{r_p}$). As the conditions $(i)-(vii)$ trivially hold, this finishes the first step.

\medskip

Assume now that we have constructed $t=\la p_0,\ldots,p_{n-1}\ra$ for some $n\ge 1$ along with $\la a_{t\uhr k}\colon\ 1\le k\le n\ra$, $\la A^{t\uhr k}_m\colon\ m\in a_{t\uhr k}, 1\le k\le n\ra$, $\la l_{t\uhr k}^{t\uhr n'}\colon\ 1\le k<n'\le n\ra$, $\la e_{t\uhr k,m}^{t\uhr n'}\colon\ m\in a_{t\uhr i},1\le i<k<n'\le n\ra$, and $\la b_{t\uhr k}^{t\uhr n'}\colon\ 1\le k<n'\le n\ra$ satisfying $(i)$-$(vii)$ whenever relevant. Assume also that for every $1\le k\le n$ we have:
\[\tag{$*$}\big|a_{t\uhr k}\big|>(k+1)^2\cdot\big(\max a_{t\uhr(k-1)}+1\big)^2\cdot2^k,\]
where we assign $\max a_\emptyset=0$.

Fix $p\in\IP$. There are $q_p\le p$ and a sequence $\la l_k^{q_p}\in a_{t\uhr k}\colon\ 1\le k\le n\ra$ such that for every $1\le k\le n$ if $q_p\forces\name{x}\not\in\bigvee_{m\in a_{t\uhr k}}A_m^{t\uhr k}$, then $l_k^{q_p}=\min a_{t\uhr k}$, and
$q_p\forces\name{x}\in A_{l_k^{q_p}}^{t\uhr k}$ otherwise. Put:
\[\A_0=\bigcup_{1\le k\le n}\big\{A_m^{t\uhr k}\colon\ m\in a_{t\uhr k},m\neq l_k^{q_p}\big\}.\]
Hence, $q_p$ forces that $\name{x}\not\in\bigvee\A_0$, and so, by Lemma \ref{lemma:aN_antichain}, there are names $\name{X}$ for an infinite subset of $\omega$ and $\name{f}$ for a map from $\name{X}$ to $\A$ whose range is an antichain with elements disjoint with $\bigvee\A_0$ and not containing $\name{x}$, such that $q_p$ forces that for every $m\in\name{X}$ and $\C\subset\A_0$ we have
\[\big|\name{\mu}_m\big(\name{f}(m)\big)\big|>\big|\name{\mu}_m\big(\bigvee\C\big)\big|+n+1.\]

There exist $r_p\le q_p$, $a_{r_p}\in[\w]^{<\w}$ of size
\[\big|a_{r_p}\big|>(n+2)^2\cdot\big(\max a_t+1\big)^2\cdot2^{n+1}\]
with $\max a_t<\min a_{r_p}$, and $\la A_m^{r_p}\in\A\colon\ m\in a_{r_p}\ra$ such that $r_p$ forces that $a_{r_p}\subset\name{X}$ and $\name{f}(m)=A_m^{r_p}$ for every $m\in a_{r_p}$ (hence $\la A_m^{r_p}\colon\ m\in a_{r_p}\ra$ is also an antichain).

Finally, there exist $s_p\le r_p$ and for every $1\le i<k\le n$ and every $m\in a_{t\uhr i}$ a set $e_{k,m}^{s_p}\subset a_{t\uhr k}$ of size $\big|e_{k,m}^{s_p}\big|\le m\cdot2^k$, such that $s_p$ forces that
\[\big\{l\in a_{t\uhr k}\colon\ \big|\name{\mu}_{m}\big|\big(A^{t\uhr k}_l\big)\ge 1/2^k\big\}=e_{k,m}^{s_p}\]
for all $k$ and $m$ as above (recall here that $1_\IP\forces\big\|\name{\mu}_m\big\|\le m$ for every $m\in\w$). 

As in the first step, the map $p\mapsto s_p$ (with the domain $\IP$) has the range which is a dense subset of $\IP$ --- call this range $D_t^T$, i.e. for all $q\in D_t^T$ the sequence $t\bigvid q$ will be in the tree $T$. Again, $D_t^T\in V$. For every $s_p\in D_t^T$ and $m\in a_{r_p}$ rename $a_{t\bigvid s_p}=a_{r_p}$, $A_m^{t\bigvid s_p}=A_m^{r_p}$ and $l_{t\uhr k}^{t\bigvid s_p}=l_k^{q_p}$. Finally, for every $s_p\in D_t^T$, $1\le i<k\le n$ and $m\in a_{t\uhr i}$ rename $e_{t\uhr k,m}^{t\bigvid s_p}=e_{k,m}^{s_p}$ and put:
\[b_{t\uhr k}^{t\bigvid s_p}=\Big\{l_{t\uhr k}^{t\bigvid s_p}\Big\}\cup\bigcup_{j=1}^{k-1}\bigcup_{m\in a_{t\uhr j}}e_{t\uhr k,m}^{t\bigvid s_p}.\]

It is easy to see that all the demanded conditions (including auxiliary $(*)$), maybe except for $(ii)$, are satisfied. To show $(ii)$, fix $s_p\in D_t^T$ and $1\le k\le n$. Write $t'=t\bigvid s_p$ for simplicity and note that $t\uhr k=t'\uhr k$. We have:
\[\big|b_{t'\uhr k}^{t'}\big|\le 1+\sum_{j=1}^{k-1}\sum_{m\in a_{t\uhr j}}\big|e_{t\uhr k,m}^{t'}\big|\le 1+\sum_{j=1}^{k-1}\sum_{m\in a_{t\uhr j}}m\cdot 2^k\le1+\sum_{j=1}^{k-1}\big(\max a_{t\uhr j}\big)^2\cdot 2^k\le\]
\[\le 1+(k-1)\cdot\big(\max a_{t\uhr(k-1)}\big)^2\cdot 2^k\le(k+1)\cdot\big(\max a_{t'\uhr(k-1)}+1\big)^2\cdot 2^k.\]
Combining this with $(*)$ we obtain:
\[\frac{\big|a_{t'\uhr k}\big|}{\big|b_{t'\uhr k}^{t'}\big|}>\frac{(k+1)^2\cdot\big(\max a_{t'\uhr(k-1)}+1\big)^2\cdot2^k}{(k+1)\cdot\big(\max a_{t'\uhr(k-1)}+1\big)^2\cdot 2^k}=k+1,\]
which yields (ii).
\end{proof}

The proof of the next proposition is obviously similar to the proof of Proposition \ref{prop:tree_nikodym}, but not identical.

\begin{proposition}\label{prop:tree_grothendieck}
Let $\A$ be a ground model Boolean algebra. Assume that $\la\name{B_n}\colon\ n\in\w\ra$ is a sequence of $\IP$-names for elements of $\A$ such that $1_\IP$ forces that $\la\name{B_n}\colon\ n\in\w\ra$ is an antichain. Assume also that $\la\name{\mu}_n\colon\ n\in\w\ra$ is a sequence of $\IP$-names for measures on $\A$ and that there exist rational numbers $M,\varepsilon>0$ such that for all $n\in\w$ we have:
\[1_\IP\forces\big\|\name{\mu}_n\big\|<M\text{ and }\big|\name{\mu}_n\big(\name{B}_n\big)\big|>2\varepsilon.\]
Then, there exists a tree $T\subset\IP^{<\omega}$ such that to every $t=\la p_0,\ldots,p_{n-1}\ra\in T$ ($n\ge1$) there are associated the following objects:
\begin{itemize}
	\item a set $a_t\in[\omega]^{<\omega}$,
	\item a sequence $\la A_m^t\in\A\colon\ m\in a_t\ra$,
	\item a sequence $\la b_{t\uhr k}^t\subset a_{t\uhr k}\colon\ 1\le k<n\ra$,
	\item a sequence $\la c_{t\uhr k}^t\subset b_{t\uhr k}^t\colon\ 1\le k<n\ra$,
	\item a sequence $\la e_{t\uhr k,m}^t\subset b_{t\uhr k}^t\colon\ m\in a_{t\uhr i},1\le i<k<n\ra$,
	\item a name $\name{X}_t$ for a subset of $\omega$,
\end{itemize}
satisfying the following conditions:
\begin{enumerate}[(i)]
	\item $\max a_{t\uhr k}<\min a_{t\uhr (k+1)}$ for all $1\le k<n$,
	\item $\big|a_{t\uhr k}\big|>(k+1)\cdot\big|b_{t\uhr k}^t\big|$ for all $1\le k<n$,
	\item $b_{t\uhr k}^t=c_{t\uhr k}^t\cup\bigcup_{i<k}\bigcup_{m\in a_{t\uhr i}}e_{t\uhr k,m}^t$ for all $1\le k<n$,
	\item $1_\IP\forces\name{X}_t\in[\omega]^\omega$ and $\name{X}_{t\uhr(k+1)}\subset\name{X}_{t\uhr k}$ for all $1\le k<n$,
	\item $1_\IP\forces\big\{l\in a_{t\uhr k}\colon\ \big|\name{\mu}_m\big|\big(A^{t\uhr k}_l\big)\ge\varepsilon/2^{k+2}\big\}=\big\{l\in a_{t\uhr k}\colon\ \big|\name{\mu}_{m'}\big|\big(A^{t\uhr k}_l\big)\ge\varepsilon/2^{k+2}\big\}$ for all $m,m'\in\name{X}_t$ and $1\le k<n$,
	\item $p_{n-1}\forces a_t\subset\name{X}_t$ and $\name{B}_m=A_m^t$ for all $m\in a_t$,
	\item $p_{n-1}\forces c_{t\uhr k}^t=\big\{l\in a_{t\uhr k}\colon\ \exists m\in\name{X}_t\text{ such that } \big|\name{\mu}_m\big|\big(A_l^{t\uhr k}\big)\ge\varepsilon/2^{k+2}\big\}$ for all $1\le k<n$,
	\item $p_{n-1}\forces e_{t\uhr k,m}^t=\big\{l\in a_{t\uhr k}\colon\ \big|\name{\mu}_m\big|\big(A_l^{t\uhr k}\big)\ge\varepsilon/2^{k+2}\big\}$ for all $m\in a_{t\uhr i}$ and $1\le i<k<n$.
\end{enumerate}
Besides, for every $t\in T$ the set $D_t^T=\big\{q\colon\ t\bigvid q\in T\big\}\in V$ is dense in $\IP$.
\end{proposition}
\begin{proof}
The first level of the tree $T$ is constructed as follows. Fix $p\in\IP$. 
There exist $r_p\le p$, $a_{r_p}\in[\w]^{<\w}$ of size $\big|a_{r_p}\big|>2^4\cdot M/\varepsilon$ and $\la A_m^{r_p}\in\A\colon\ m\in a_{r_p}\ra$ such that $r_p$ forces that $\name{B}_m=A_m^{r_p}$ for every $m\in a_{r_p}$.

The map $p\mapsto r_p$ (with the domain $\IP$) has the range which is a dense subset of $\IP$ --- call this range $D_\emptyset^T$, i.e. $D_\emptyset^T$ will be the first level of the tree $T$. Note that trivially $D_\emptyset^T\in V$. For every $r_p\in D_\emptyset^T$ and $m\in a_{r_p}$ rename $a_{\la r_p\ra}=a_{r_p}$ and $A_m^{\la r_p\ra}=A_m^{r_p}$ (if for different $p$ and $p'$ we have got $r_p=r_{p'}$, then assume that $a_{r_p}=a_{r_{p'}}$ and $A_m^{r_p}=A_m^{r_{p'}}$ for all $m\in a_{r_p}$) and, finally, let $\name{X}_{\la r_p\ra}$ be a name for $\omega$. As the conditions $(i)-(viii)$ trivially hold, this finishes the first step.

\medskip

Assume now that we have constructed $t=\la p_0,\ldots,p_{n-1}\ra$ for some $n\ge 1$ along with $\la a_{t\uhr k}\colon\ 1\le k\le n\ra$, $\la\name{X}_{t\uhr k}\colon\ k\le n\ra$, $\la A^{t\uhr k}_m\colon\ m\in a_{t\uhr k}, 1\le k\le n\ra$, $\la e_{t\uhr k,m}^{t\uhr n'}\colon\ m\in a_{t\uhr i},1\le i<k<n'\le n\ra$, $\la c_{t\uhr k}^{t\uhr n'}\colon\ 1\le k<n'\le n\ra$, and $\la b_{t\uhr k}^{t\uhr n'}\colon\ 1\le k<n'\le n\ra$ satisfying $(i)$-$(viii)$ whenever relevant. Assume also that for every $1\le k\le n$ we have:
\[\tag{$*$}\big|a_{t\uhr k}\big|>(k+1)\cdot\big(\max a_{t\uhr(k-1)}+1\big)^2\cdot M\cdot2^{k+2}/\varepsilon,\]
where we assign $\max a_\emptyset=0$.

Fix $p\in\IP$. Since for every $1\le k\le n$ the set $a_{t\uhr k}$ is finite, there is a name $\name{X}$ for an infinite subset of $\omega$ such that $1_\IP$ forces that $\name{X}\subset\name{X}_t$ and
\[\big\{l\in a_{t\uhr k}\colon\ \big|\name{\mu}_m\big|\big(A^{t\uhr k}_l\big)\ge\varepsilon/2^{k+2}\big\}=\big\{l\in a_{t\uhr k}\colon\ \big|\name{\mu}_{m'}\big|\big(A^{t\uhr k}_l\big)\ge\varepsilon/2^{k+2}\big\}\]
for every $m,m'\in\name{X}$ and $1\le k\le n$. Since for every $n\in\w$ we have $1_\IP\forces\big\|\name{\mu}_n\big\|<M$, there exist $q_p\le p$ and $c_k^{q_p}\subset a_{t\uhr k}$ of size $\big|c_k^{q_p}\big|\le M\cdot2^{k+2}/\varepsilon$ such that $q_p$ forces for every $m\in\name{X}$ that
\[\big\{l\in a_{t\uhr k}\colon\ \big|\name{\mu}_m\big|\big(A^{t\uhr k}_l\big)\ge\varepsilon/2^{k+2}\big\}=c_k^{q_p}.\] 

There exist $r_p\le q_p$, $a_{r_p}\in[\w]^{<\w}$ of size
\[\big|a_{r_p}\big|>(n+2)\cdot\big(\max a_t+1\big)^2\cdot M\cdot2^{n+3}/\varepsilon\]
with $\max a_t<\min a_{r_p}$, and $\la A_m^{r_p}\in\A\colon\ m\in a_{r_p}\ra$ such that $r_p$ forces that $a_{r_p}\subset\name{X}$ and $\name{B}_m=A_m^{r_p}$ for every $m\in a_{r_p}$.

Finally, there exist $s_p\le r_p$ and for every $1\le i<k\le n$ and every $m\in a_{t\uhr i}$ a set $e_{k,m}^{s_p}\subset a_{t\uhr k}$ of size $\big|e_{k,m}^{s_p}\big|\le M\cdot2^{k+2}/\varepsilon$, such that $s_p$ forces that
\[\big\{l\in a_{t\uhr k}\colon\ \big|\name{\mu}_{m}\big|\big(A^{t\uhr k}_l\big)\ge\varepsilon/2^{k+2}\big\}=e_{k,m}^{s_p}\]
for all $k$ and $m$ as above (recall here that $1_\IP\forces\big\|\name{\mu}_m\big\|<M$ for every $m\in\w$).

As in the first step, the map $p\mapsto s_p$ (with the domain $\IP$) has the range which is a dense subset of $\IP$ --- call this range $D_t^T$, i.e. for all $q\in D_t^T$ the sequence $t\bigvid q$ will be in the tree $T$. Again, $D_t^T\in V$. For every $s_p\in D_t^T$ and $m\in a_{r_p}$ rename $\name{X}_{t\bigvid s_p}=\name{X}$, $a_{t\bigvid s_p}=a_{r_p}$, $A_m^{t\bigvid s_p}=A_m^{r_p}$ and $c_{t\uhr k}^{t\bigvid s_p}=c_k^{q_p}$. Finally, for every $s_p\in D_t^T$, $1\le i<k\le n$ and $m\in a_{t\uhr i}$ rename $e_{t\uhr k,m}^{t\bigvid s_p}=e_{k,m}^{s_p}$ and put:
\[b_{t\uhr k}^{t\bigvid s_p}=c_{t\uhr k}^{t\bigvid s_p}\cup\bigcup_{1\le j<k}\bigcup_{m\in a_{t\uhr j}}e_{t\uhr k,m}^{t\bigvid s_p}.\]

It is easy to see that all the demanded conditions (including auxiliary $(*)$), maybe except for $(ii)$, are satisfied. To show $(ii)$, fix $s_p\in D_t^T$ and $1\le k\le n$. Write $t'=t\bigvid s_p$ for simplicity and note that $t\uhr k=t'\uhr k$. We have:
\[\big|b_{t'\uhr k}^{t'}\big|\le\big|c_{t\uhr k}^{t'}\big|+\sum_{j=1}^{k-1}\sum_{m\in a_{t\uhr j}}\big|e_{t\uhr k,m}^{t'}\big|\le M\cdot2^{k+2}/\varepsilon+\sum_{j=1}^{k-1}\sum_{m\in a_{t\uhr j}}M\cdot2^{k+2}/\varepsilon\le\]
\[\le\Big(1+\sum_{j=1}^{k-1}\big|a_{t\uhr j}\big|\Big)\cdot M\cdot2^{k+2}/\varepsilon\le\Big(1+\sum_{j=1}^{k-1}\max a_{t\uhr j}\Big)\cdot M\cdot2^{k+2}/\varepsilon\le\]
\[\le\Big(1+\big(\max a_{t\uhr(k-1)}\big)^2\Big)\cdot M\cdot2^{k+2}/\varepsilon\le\big(1+\max a_{t'\uhr(k-1)}\big)^2\cdot M\cdot2^{k+2}/\varepsilon.\]
Combining this with $(*)$ we obtain:
\[\frac{\big|a_{t'\uhr k}\big|}{\big|b_{t'\uhr k}^{t'}\big|}>\frac{(k+1)\cdot\big(\max a_{t'\uhr(k-1)}+1\big)^2\cdot M\cdot2^{k+2}/\varepsilon}{\big(\max a_{t'\uhr(k-1)}+1\big)^2\cdot M\cdot2^{k+2}/\varepsilon}=k+1,\]
which yields (ii).
\end{proof}

Let us briefly compare the tree, say $T_N$, from Proposition \ref{prop:tree_nikodym} and the tree, say $T_G$, from Proposition \ref{prop:tree_grothendieck}. The main difference lays in the construction of the associated collections $\la A_m^t\in\A\colon\ m\in a_t\ra$. Namely, in each step of the construction of $T_N$ we just found such a collection using Lemma \ref{lemma:aN_antichain} --- the only necessary condition was that elements of the collection must be pairwise disjoint and disjoint from the element $\bigvee\A_0$ (defined as the union of the elements from the collections constructed in the previous steps) and must satisfy a certain measure-theoretic inequality, while in the construction of $T_G$ the elements of $\la A_m^t\colon\ m\in a_t\ra$ were found as approximating the already fixed antichain $\la\name{B}_n^G\colon\ n\in\w\ra$ in $V[G]$ (given by Lemma \ref{lemma:aG_antichain}) --- we had much less  freedom of choice than in the case of $T_N$. This difference had an important effect: in the case of $T_N$ we could keep Proposition \ref{prop:tree_nikodym}(iv) true and finally obtain an infinite antichain (see Remark \ref{remark:branch_generic_infinite}), while in the case of $T_G$ a similar condition was not possible to obtain --- collections from distinct steps may have had non-disjoint elements.

\medskip

\begin{lemma}\label{lemma:branch_compatible}
Let $\vec{p}=\la p_n\colon\ n\in\w\ra\in\IP^\omega$. Let $k<n,n'\in\w$ and assume that $p_{n-1}$ and $p_{n'-1}$ are compatible. Then:
\begin{enumerate}
	\item if $\vec{p}$ is a branch of the tree $T$ from Proposition \ref{prop:tree_nikodym} or \ref{prop:tree_grothendieck}, then $e_{\vec{p}\uhr k,m}^{\vec{p}\uhr n}=e_{\vec{p}\uhr k,m}^{\vec{p}\uhr n'}$ for all $m\in a_{\vec{p}\uhr i}$ and $i<k$;
	\item if $\vec{p}$ is a branch of the tree $T$ from Proposition \ref{prop:tree_nikodym}, then $l_{\vec{p}\uhr k}^{\vec{p}\uhr n}=l_{\vec{p}\uhr k}^{\vec{p}\uhr n'}$;
	\item if $\vec{p}$ is a branch of the tree $T$ from Proposition \ref{prop:tree_grothendieck}, then $c_{\vec{p}\uhr k}^{\vec{p}\uhr n}=c_{\vec{p}\uhr k}^{\vec{p}\uhr n'}$.
\end{enumerate}
In particular, $b_{\vec{p}\uhr k}^{\vec{p}\uhr n}=b_{\vec{p}\uhr k}^{\vec{p}\uhr n'}$ in both propositions.
\end{lemma}
\begin{proof}
(1) follows from Proposition \ref{prop:tree_nikodym}(v) and Proposition \ref{prop:tree_grothendieck}(viii). (2) holds by Proposition \ref{prop:tree_nikodym}(vi). (3) follows from Proposition \ref{prop:tree_grothendieck}(iv--v) and \ref{prop:tree_grothendieck}(vii). The last sentence follows from (iii) of both propositions.
\end{proof}

\begin{remark}\label{remark:branch_generic_infinite}
We work in $V[G]$. Let $\vec{p}=\la p_n\colon\ n\in\w\ra\in\IP^\omega$ be a branch of the tree $T$ from Proposition \ref{prop:tree_nikodym} or \ref{prop:tree_grothendieck}. Assume that the set $I=\big\{n\colon\ p_{n-1}\in G\}$ is infinite. Since $G\subset\IP$ is a filter, it follows from Lemma \ref{lemma:branch_compatible} that we have $b_{\vec{p}\uhr k}^{\vec{p}\uhr n}=b_{\vec{p}\uhr k}^{\vec{p}\uhr n'}$ whenever $n,n'\in I$ and $k\in\w$ is such that $1\le k<n,n'$. Hence, for every $k\in\w\setminus\{0\}$ we can put $b_k=b_{\vec{p}\uhr k}^{\vec{p}\uhr n}$, where $n\in I$ is arbitrary and $n>k$. We put also $a_0=\emptyset$ and $a_k=a_{\vec{p}\uhr k}$ for $k>0$, and since $\la a_k\colon\ k\in\w\ra$ is a sequence of pairwise disjoint sets (by (i) in both propositions), we may put $A_m=A_m^{\vec{p}\uhr k}$ for every $m\in a_k$ and $k\in\w$.
\end{remark}

\begin{lemma}\label{lemma:generic_branch_estimation}
Let $\vec{p}=\la p_n\colon\ n\in\w\ra\in\IP^\omega\cap V[G]$. Assume that $I=\big\{n\colon\ p_{n-1}\in G\}$ is infinite and let $\la m_k\colon\ k\in\w\setminus\{0\}\ra\in V$ be a sequence such that $m_k\in a_k\setminus b_k$ for all $k\in\w\setminus\{0\}$. Fix $n\in I$. Then, in $V[G]$ the following hold:
\begin{enumerate}
	\item if $\vec{p}$ is a branch of the tree $T$ from Proposition \ref{prop:tree_nikodym}, then $\big|\name{\mu}_{m_n}^G\big|\big(A_{m_k}\big)<1/2^k$ for all $k\in\w$ such that $k>n$;
	\item if $\vec{p}$ is a branch of the tree $T$ from Proposition \ref{prop:tree_grothendieck}, then $\big|\name{\mu}_{m_n}^G\big|\big(A_{m_k}\big)<\varepsilon/2^{k+2}$ for all $k\neq 0,n$.
\end{enumerate}
\end{lemma}
\begin{proof}
(1) Fix $k\in\w$ such that $k>n$. Since $I$ is infinite, there is $n'\in I$ such that $n'-1>k$. By Proposition \ref{prop:tree_nikodym}(v) we have:
\[p_{n'-1}\forces e_{\vec{p}\uhr k,m_n}^{\vec{p}\uhr n'}=\big\{l\in a_{\vec{p}\uhr k}\colon\ \big|\name{\mu}_{m_n}\big|\big(A_l^{\vec{p}\uhr k}\big)\ge 1/2^k\big\}.\]
Since $m_k\not\in b_k=b_{\vec{p}\uhr k}^{\vec{p}\uhr n'}\supset e_{\vec{p}\uhr k,m_n}^{\vec{p}\uhr n'}$ and $p_{n'-1}\in G$, we have:
\[\big|\name{\mu}_{m_n}^G\big|\big(A_{m_k}\big)=\big|\name{\mu}_{m_n}^G\big|\big(A_{m_k}^{\vec{p}\uhr k}\big)<1/2^k.\]

(2) Fix $k\in\w\setminus\{0\}$. If $k>n$, then similarly as in (1), using Proposition \ref{prop:tree_grothendieck}(viii), we show that $\big|\name{\mu}_{m_n}^G\big|\big(A_{m_k}\big)<\varepsilon/2^{k+2}$.

Suppose now that $k<n$. By Proposition \ref{prop:tree_grothendieck}(vi) we have:
\[p_{n-1}\forces c_{\vec{p}\uhr k}^{\vec{p}\uhr n}=\big\{l\in a_{\vec{p}\uhr k}\colon\ \exists m\in\name{X}_{\vec{p}\uhr n}\text{ such that }\big|\name{\mu}_m\big|\big(A_l^{\vec{p}\uhr k}\big)\ge\varepsilon/2^{k+2}\big\}.\]
Since $p_{n-1}\in G$, it follows that $m_n\in a_n=a_{\vec{p}\uhr n}\subset\name{X}_{\vec{p}\uhr n}^G$, and hence by the fact that $m_k\not\in b_k=b_{\vec{p}\uhr k}^{\vec{p}\uhr n}\supset c_{\vec{p}\uhr k}^{\vec{p}\uhr n}$, we obtain:
\[\big|\name{\mu}_{m_n}^G\big|\big(A_{m_k}\big)=\big|\name{\mu}_{m_n}^G\big|\big(A_{m_k}^{\vec{p}\uhr k}\big)<\varepsilon/2^{k+2}.\]
\end{proof}

\section{Auxiliary set-theoretic results\label{section:aux_set_theory}}

In this section we present some combinatorial results implied by the preservation of the ground model set of reals as a non-meager subset of the reals in the extension.

\begin{definition}\label{def:preservation_nonmeager}
A poset $\IP$ \textit{preserves the ground model reals non-meager} if the set $\IR^V$ is a non-meager subset of $\IR^{V[G]}$ for any $\IP$-generic filter $G$.
\end{definition}

\noindent Typical examples of notions of forcing preserving the ground model reals non-meager include Sacks, side-by-side products of Sacks, Miller, and Silver (see Raghavan \cite[Section 5]{Rag09}). The property is preserved by countable support iterations (\cite[Theorem 61]{Rag09}).

The following lemma is the only place in this paper where we need properness. Recall that $H(\theta)$ for a regular cardinal number $\theta$ denotes the family of all subsets of hereditary cardinality $<\theta$.

\begin{lemma}\label{lemma:generic_branch}
Let $\IP$ be a proper forcing preserving the ground model reals non-meager. Let $M$ be a countable elementary submodel of $H(\theta)$ for some regular cardinal number $\theta$. Then, there exists a $\IP$-generic filter $G$ over $V$ having the following property: for every tree $T\subset\IP^{<\w}$ in $V\cap M$ such that for every $t\in T\cap M$ the set of successors $D_t=\big\{q\colon\ t\bigvid q\in T\big\}$ is dense in $\IP$ and $D_t\in V\cap M$, there exists a branch $\vec{p}=\la p_n\colon\ n\in\w\ra\in V$ of $T$ such that $I=\big\{n\colon\ p_{n-1}\in G\big\}$ is infinite.
\end{lemma}
\begin{proof}
Let $q_0$ be an $(M,\IP)$-generic condition. Let $T\subset\IP^{<\w}$ be a tree in $V\cap M$ such that for every $t\in T\cap M$ the set $D_t=\big\{q\in\IP\colon\ t\bigvid q\in T\big\}$ is a dense subset of $\IP$ and $D_t\in V\cap M$.
Set $T_0=T\cap M$ and note that $D_t\cap M$ is predense below $q_0$ for all $t\in T_0$. Let $G$ be a $\IP$-generic filter containing $q_0$.

In what follows we work in $V[G]$. Let $[T_0]$ denote the set of all branches of $T_0$. Suppose, contrary to our claim, that
$[T_0]\cap V\subset\bigcup_{k\in\w}X_k$, where
\[X_k=\big\{\la p_n\colon\ n\in\w\ra\in [T_0]\colon\ p_n\not\in G\text{ for all }n\ge k\big\}.\]
Note that $X_k$ is closed in $[T_0]$
for all $k\in\w$ and $[T_0]\cap V$ is non-meager in $[T_0]$ (since $\IP$ preserves the ground model reals non-meager), and hence there exists $t\in T_0$
and $k\in\w$ such that $|t|>k$ and $U_t\cap V\subset X_k$, where the set
\[U_t=\big\{\vec{p}\in [T_0]\colon\ \vec{p}\uhr |t|=t\big\}\]
is the basic open subset of $[T_0]$ generated by $t$. 
Since  $D_t\cap M$ is predense below $q_0\in G$,
there exists $p\in D_t\cap M\cap G$. Note that $t\bigvid p\in T_0$. Let $\vec{p}$ in $U_t\cap V$ such that $\vec{p}\big(|t|\big)=p$.
Then $\vec{p}\in (U_t\cap V)\setminus X_k$, a contradiction.
\end{proof}

\subsection{Almost disjoint families}

The following two lemmas seem folklore.

\begin{lemma}\label{lemma:ad_aux}
Suppose that $\IP$ preserves the ground model reals non-meager, 
$G$ is a $\IP$-generic filter and $I\in [\w]^{\w}\cap V[G]$.
For any sequence $\la H_k:k\in\w\ra \in V$ of mutually disjoint infinite subsets of
$\w$ such that $I\cap H_k$ is non-empty for every $k\in\w$, there exists a function
$f\in\w^\w\cap V$ such that the set
\[I\cap\bigcup_{k\in\w}(H_k\cap f(k))\]
is infinite.
\end{lemma}
\begin{proof}
We work in $V[G]$. Set $g(n)=\min(I\cap H_n)$. Since  $\IP$  preserves the ground model reals non-meager,
it obviously cannot add dominating functions, and hence there exists
$f\in\w^\w\cap V$ such that $g(n)<f(n)$ for infinitely many $n$.
It is easy to see that the set
\[I\cap\bigcup_{k\in\w}(H_k\cap f(k))\]
is infinite, which finishes the proof.
\end{proof}

\begin{lemma}\label{lemma:ad}
Suppose that $\IP$ preserves the ground model reals non-meager, 
$G$ is a $\IP$-generic filter and $I\in [\w]^{\w}\cap V[G]$. Then, there exists an uncountable almost
disjoint family  $\mathcal H\subset [\w]^\w\cap V$, $\mathcal{H}\in V$, such that $H\cap I$ is infinite for all $H\in\mathcal H$.
\end{lemma}
\begin{proof}
Throughout the whole proof we work in $V[G]$.
First let us note that there exists a decomposition $\w=\sqcup_{n\in\w}B_n$ of $\w$
such that $\la B_n\colon\ n\in\w\ra\in V$ and $\big|B_n\cap I\big|=\w$ for all $n\in\w$. Indeed,
fix any decomposition $\w=\sqcup_{n\in\w}C_n$ of $\w$ into infinite sets such that
 $\la C_n\colon\ n\in\w\ra\in V$, and consider subsets
\[S_n=\{\sigma\in \mathit{Sym}(\w)\colon\ \big|\sigma[C_n]\cap I\big|<\w\}\]
of the symmetric group $\mathit{Sym}(\w)$ of all permutations of $\w$. It is easy to see that
each $S_n$ is a meager subset of $\mathit{Sym}(\w)$, and hence, by our assumption on $\IP$, there exists a permutation
\[\sigma\in\big(\mathit{Sym}(\w)\cap V\big)\setminus\bigcup_{n\in\w}S_n.\]
Set $B_n=\sigma[C_n]$.

Fix a family $\{D_\tau\colon\ \tau\in 2^{<\w}\}\in V$ of infinite subsets of $\w$
such that $D_\emptyset=\w$ and $D_{\tau}=D_{\tau\vid 0}\cup D_{\tau\vid 1}$ for all $\tau\in 2^{<\w}$.
For every $x\in 2^\w\cap V$ consider the sequence
$\la H^x_k\colon\ k\in\w\ra$, where
\[H^x_k=\bigcup\big\{B_n\colon\ n\in D_{x\uhr k}\setminus D_{x\uhr(k+1)}\big\}.\]
Observe that $\la H^x_k\colon\ k\in\w\ra\in V$ for all $x\in 2^\w\cap V$ and
$H^x_{k_1}\cap H^x_{k_2}=\emptyset$ for all $k_1\neq k_2$. Moreover,
if  $x,  y\in  2^\w\cap V$ and $x(k)\neq y(k)$ for some $k$, then
$H^x_{k_1}\cap H^y_{k_2}=\emptyset$ for any $k_1,k_2\ge k$.

Since for any $x\in 2^\w\cap V$ and $k\in\w$ there exists $n\in\w$
such that $B_n\subset H^x_k$, the sequence $\la H^x_k\colon\ k\in\w\ra $ satisfies the
requirements of Lemma \ref{lemma:ad_aux}, and hence there exists $f^x\in\w^\w\cap V$
such that  $\big|I\cap H^x\big|=\w$, where
\[H^x=\bigcup_{k\in\w}(H^x_k\cap f^x(k))\in V.\]
Since $\big|H^x\cap H^y\big|<\w$ for any $x\neq y$, the family
$\mathcal H=\{H^x\colon\ x\in 2^\w\cap V\}$ is as required.
\end{proof}

\section{Main result\label{section:main}}

Recall the following standard definition.

\begin{definition}\label{def:laver_property}
A poset $\IP$ has \textit{the Laver property} if for any $\IP$-generic filter $G$, functions $f\in\w^\w\cap V$ and $g\in\w^\w\cap V[G]$ such that $f$ eventually dominates $g$ ($g\le^*f$), there exists in $V$ a function $H\colon\w\to[\w]^{<\w}$ such that $|H(n)|\le n+1$ and $g(n)\in H(n)$ for all $n\in\w$.
\end{definition}

\noindent The following standard proper posets have the Laver property: Sacks and side-by-side products of Sacks (Bartoszy\'nski and Judah  \cite[Lemma 6.3.38]{BarJud95}), Laver, Mathias, Miller (\cite[Section 7.3]{BarJud95}), and Silver (more generally, Silver-like posets) (Halbeisen \cite[Chapter 22]{Hal12}). The Laver property is preserved by countable support iterations (\cite[Theorem 6.3.34]{BarJud95}). For more information about the property see e.g. Bartoszy\'nski and Judah \cite[Section 6.3.E]{BarJud95} or Halbeisen \cite[Chapter 20]{Hal12}.

\begin{remark}\label{remark:laver}
Note that if a poset $\IP$ has the Laver property, then it has the following property. Let $f\in\w^\w\cap V$ and let $\la\mathcal{F}_k\colon\ k\in\w\ra\in V$ be a sequence of finite sets such that $|\mathcal{F}_k|=f(k)$ for every $k\in\w$. Let $G$ be a $\IP$-generic filter over $V$ and let $b\in\prod_{k\in\w}\mathcal{F}_k$ be in $V[G]$. Then, there exists a function
\[B\in\prod_{k\in\w}\big[\mathcal{F}_k\big]^{\le k+1}\]
in $V$ such that $b(k)\in B(k)$ for every $k\in\w$.
\end{remark}

We are now in the position to prove the main theorem of the paper.

\begin{theorem}\label{theorem:main}
Let $\IP\in V$ be a notion of proper forcing having the Laver property and preserving the ground model reals non-meager. Let $\A\in V$ be a $\sigma$-complete Boolean algebra. Then, for every $\IP$-generic filter $G$ over $V$ the algebra $\A$ has the Vitali--Hahn--Saks property in $V[G]$.
\end{theorem}
\begin{proof}
Let $G$ be a $\IP$-generic filter over $V$. To show that $\A$ has the Vitali--Hahn--Saks property in $V[G]$, we prove that it has in $V[G]$ both the Nikodym property and the Grothendieck property. The proof will follow by the \textit{a contrario} argument.

\medskip

\noindent\textbf{Case (N). } If $\A$ does not have the Nikodym property in $V[G]$, then there exists an anti-Nikodym sequence of measures $\la\mu_n\colon\ n\in\w\ra\in V[G]$. Without loss of generality we may assume that $\big\|\mu_n\big\|<n$ for every $n\in\w$ (if $\big\|\mu_n\big\|\ge n$ for some $n\in\w$, then replace it with $0.5n\cdot\mu_n/\big\|\mu_n\big\|$). Let $x\in K_\A\cap V[G]$ be a Nikodym concentration point of $\la\mu_n\colon\ n\in\w\ra$.

In $V$, we may assume that $1_\IP$ forces that $\la\name{\mu}_n\colon\ n\in\w\ra$ is anti-Nikodym,  $\name{x}$ is its Nikodym concentration point, and $\big\|\name{\mu}_n\big\|<n$ for every $n\in\w$.

\medskip

\noindent\textbf{Case (G). } If $\A$ does not have the Grothendieck property in $V[G]$, then there exist an anti-Grothendieck sequence of measures $\la\mu_n'\colon\ n\in\w\ra\in V[G]$, norm bounded by some rational number $M$, and, by Lemma \ref{lemma:aG_antichain}, an antichain $\la B_n\in\A\colon\ n\in\w\ra\in V[G]$ and rational $\varepsilon>0$ such that $\big|\mu_n'\big(B_n\big)\big|>2\varepsilon$ for every $n\in\w$.

In $V$, we may assume that $1_\IP$ forces that $\la\name{\mu}_n'\colon\ n\in\w\ra$ is anti-Grothendieck, $\la\name{B}_n\in\A\colon\ n\in\w\ra$ is an antichain, $\big\|\name{\mu}_n'\big\|<M$ and $\big|\name{\mu}_n'\big(\name{B}_n)\big|>2\varepsilon$ for every $n\in\w$. 

\medskip

\noindent\textbf{Cases (N) and (G). } For a moment, the proof goes simultaneously for both Case (N) and Case (G).

\medskip

Let $T\subset\IP^{<\w}$ be a tree in $V$ from Proposition \ref{prop:tree_nikodym} (in Case (N)) or Proposition \ref{prop:tree_grothendieck} (in Case (G)) with all the associated objects like $\big\{a_t\colon\ t\in T\setminus\{\emptyset\}\big\}$, $\big\{D_t^T\colon\ t\in T\setminus\{\emptyset\}\big\}$ etc. Let $M$ be a countable elementary submodel of $H(\theta)$ for a sufficiently big regular cardinal number $\theta$ containing $\IP$ and all the objects mentioned above. Let $G'$ be a $\IP$-generic filter from Lemma \ref{lemma:generic_branch} and $\vec{p}=\la p_n\colon\ n\in\w\ra$ an infinite branch of $T\cap V$ such that the set $I=\big\{n\colon\ p_{n-1}\in G'\big\}$ is infinite. Let $\HH\subset\big[\w\setminus\{0\}\big]^\w\cap V$, $\HH\in V$, be an almost disjoint family from Lemma \ref{lemma:ad}.

We now and to the end of the proof work in $V[G']$. Let $\la a_k\colon\ k\in\w\ra$, $\la b_k\colon\ k\in\w\setminus\{0\}\ra$ and $\la A_m\colon\ m\in a_k,k\in\w\ra$ be sequences from Remark \ref{remark:branch_generic_infinite}. Define a function $b\colon\w\setminus\{0\}\to[\w]^{<\w}$ by putting $b(k)=b_k$. Then, $b\in V[G']$. For every $k\in\w\setminus\{0\}$ put:
\[\mathcal{F}_k=\big\{c\subset a_k\colon\ \big|a_k\big|>(k+1)|c|\big\}.\]
Obviously, each $\mathcal{F}_k\in\big[[\w]^{<\w}\big]^{<\w}$ and $\la\mathcal{F}_k\colon\ k\in\w\setminus\{0\}\ra\in V$ (since $\la a_k\colon\ k\in\w\ra\in V$). Note that by (ii) in Propositions \ref{prop:tree_nikodym} and \ref{prop:tree_grothendieck} we have $b(k)\in\mathcal{F}_k$ for all $k\in\w\setminus\{0\}$. Since $\IP$ has the Laver property, by Remark \ref{remark:laver}, there exists a function 
\[B\colon\w\setminus\{0\}\to\big[[\w]^{<\w}\big]^{<\w}\]
in $V$ such that for every $k\in\w\setminus\{0\}$ the following hold:
\begin{itemize}
	\item $B(k)\subset\mathcal{F}_k$,
	\item $|B(k)|\le k+1$,
	\item $b(k)\in B(k)$.
\end{itemize}
It follows that $a_k\setminus\bigcup B(k)\neq\emptyset$ for all $k\in\w\setminus\{0\}$. Since $\la a_k\setminus\bigcup B(k)\colon\ k\in\w\ra\in V$, there exists $\la m_k\colon\ k\in\w\setminus\{0\}\ra\in V$ such that $m_k\in a_k\setminus\bigcup B(k)$ for every $k\in\w\setminus\{0\}$.


\medskip

We now again deal separately with Cases (N) and (G).

\medskip

\noindent\textbf{Case (N). } Note that $\la A_{m_k}\colon\ k\in\w\setminus\{0\}\ra\in V$ and by Proposition \ref{prop:tree_nikodym}(iv) it is an antichain. Since $\A$ is $\sigma$-complete in $V$, $\bigvee_{k\in H}A_{m_k}\in\A$ for every $H\in\HH$. By Lemma \ref{lemma:measures_ad} there exists $H_0\in\HH$ such that
\[\mu_m\Big(\bigvee_{k\in H_0}A_{m_k}\Big)=\sum_{k\in H_0}\mu_m\big(A_{m_k}\big)\]
for every $m\in\w$. Let $n\in I\cap H_0$. Note that $p_{n-1}\in G'$, so we have:
\[\big|\mu_{m_n}\Big(\bigvee_{k\in H_0}A_{m_k}\Big)\big|=\big|\sum_{k\in H_0}\mu_{m_n}\big(A_{m_k}\big)\big|\ge\]
\[\big|\mu_{m_n}\big(A_{m_n}\big)\big|-\big|\sum_{\substack{k\in H_0\\k<n}}\mu_{m_n}\big(A_{m_k}\big)\big|-\sum_{\substack{k\in H_0\\k>n}}\big|\mu_{m_n}\big|\big(A_{m_k}\big)\ge\]
\[n-\sum_{\substack{k\in H_0\\k>n}}1/2^k>n-1,\]
where the last line follows from Proposition \ref{prop:tree_nikodym}(vii) and Lemma \ref{lemma:generic_branch_estimation}(1). Thus, we get that
\[\sup_{n\in\w}\big|\mu_n\Big(\bigvee_{k\in H_0}A_{m_k}\Big)\big|=\infty,\]
which contradicts the pointwise boundedness of $\la\mu_n\colon\ n\in\w\ra$ and hence proves that $\A$ has the Nikodym property in $V[G]$.
\medskip

\noindent\textbf{Case (G). } For every $k\in\w\setminus\{0\}$ put:
\[C_k=A_{m_k}\setminus\bigvee_{0<i<k}A_{m_i}.\]
Then, $\la C_k\colon\ k\in\w\setminus\{0\}\ra$ is an antichain and, since $\la A_{m_k}\colon\ k\in\w\setminus\{0\}\ra\in V$, $\la C_k\colon\ k\in\w\setminus\{0\}\ra\in V$ as well. Note that if $n\in I$, then $p_{n-1}\in G'$ and hence $A_{m_n}=B_{m_n}$ by Proposition \ref{prop:tree_grothendieck}(vi). Again, since $\A$ is $\sigma$-complete in $V$, $\bigvee_{k\in H}C_{m_k}\in\A$ for every $H\in\HH$, so by Lemma \ref{lemma:measures_ad} there exists $H_0\in\HH$ such that
\[\mu_m\Big(\bigvee_{k\in H_0}C_{m_k}\Big)=\sum_{k\in H_0}\mu_m\big(C_{m_k}\big)\]
for every $m\in\w$. Let $n\in I\cap H_0$. Since for every $k\in\w\setminus\{0,n\}$ we have $C_k\subset A_{m_k}$, by Lemma \ref{lemma:generic_branch_estimation}(2) we also have:
\[\big|\mu_{m_n}\big|\big(C_k\big)<\varepsilon/2^{k+2}.\]
Finally, we obtain:
\[\big|\mu_{m_n}\Big(\bigvee_{k\in H_0}C_k\Big)\big|=\big|\sum_{k\in H_0}\mu_{m_n}\big(C_k\big)\big|\ge\big|\mu_{m_n}\big(C_n\big)\big|-\big|\sum_{\substack{k\in H_0\\k\neq n}}\mu_{m_n}\big(C_k\big)\big|\ge\]
\[\big|\mu_{m_n}\Big(A_{m_n}\setminus\bigvee_{0<i<n}A_{m_i}\Big)\big|-\big|\sum_{\substack{k\in H_0\\k\neq n}}\mu_{m_n}\big(C_k\big)\big|\ge\]
\[\big|\mu_{m_n}\big(A_{m_n}\big)\big|-\sum_{0<i<n}\big|\mu_{m_n}\big|\big(A_{m_i}\big)-\sum_{\substack{k\in H_0\\k\neq n}}\big|\mu_{m_n}\big|\big(C_k\big)>\]
\[2\varepsilon-\sum_{0<i<n}\varepsilon/2^{i+2}-\sum_{\substack{k\in H_0\\k\neq n}}\varepsilon/2^{k+2}>2\varepsilon-\varepsilon/2-\varepsilon/2=\varepsilon,\]
where the last line again follows from Lemma \ref{lemma:generic_branch_estimation}(2). Thus, we get that
\[\limsup_{n\to\infty}\mu_n\Big(\bigvee_{k\in H_0}C_{m_k}\Big)\ge\varepsilon>0,\]
which contradicts the fact that $\la\mu_n\colon\ n\in\w\ra$ is weak* convergent to $0$ and hence proves that $\A$ has the Gronthendieck property in $V[G]$.
\end{proof}

As mentioned in the introduction, Theorem \ref{theorem:main} gives a generalization of the results of Brech \cite{Bre06} and the authors \cite{SobZdo17} stating together that side-by-side products of the Sacks forcing preserve the Vitali--Hahn--Saks property of ground model $\sigma$-complete Boolean algebras.

\begin{corollary}\label{cor:main}
Let $\IP\in V$ be one of the following posets: Sacks forcing, side-by-side product of the Sacks forcing, Silver forcing, Miller forcing, or the countable support iteration of length $\omega_2$ of any of them. Let $\A\in V$ be a $\sigma$-complete Boolean algebra. Then, for any $\IP$-generic filter $G$ over $V$, the algebra $\A$ has the Vitali--Hahn--Saks property in $V[G]$.
\end{corollary}

No infinite Boolean algebra of cardinality strictly less than the bounding number $\mathfrak{b}$ has the Nikodym property (Sobota \cite[Proposition 3.2]{Sob17}), so if the Continuum Hypothesis holds in $V$ but $\omega_1<\mathfrak{b}=\cc$ in the generic extension $V[G]$ of a proper forcing, then no ground model $\sigma$-complete Boolean algebra of cardinality $\w_1$ has the Nikodym property in $V[G]$. This implies that in case, e.g., of the Laver model we may only ask about the Grothendieck property.

\begin{question}\label{question:laver}
Let $\A\in V$ be a $\sigma$-complete Boolean algebra. Does $\A$ have the Grothendieck property in the model obtained by the countable support iteration of length $\w_2$ of the Laver forcing?
\end{question}

Note that if the answer to Question \ref{question:laver} is positive, then we obtain a consistent example of a whole class of Boolean algebras with the Grothendieck property but without the Nikodym property. This would shed new light on such Boolean algebras, since so far only one example has been found (under the Continuum Hypothesis) --- see Talagrand \cite{Tal84}.

\medskip

It seems that changing \textit{mutatis mutandis} its proof, Theorem \ref{theorem:main} also holds for Boolean algebras with the Subsequential Completeness Property introduced by Haydon \cite{Hay81}: a Boolean algebra $\A$ has \textit{the Subsequential Completeness Property (SCP)} if for every antichain $\la A_n\colon\ n\in\w\ra$ in $\A$ there exists $M\in[\w]^\w$ such that $\bigvee_{n\in M}A_n\in\A$. Haydon \cite{Hay81} proved that algebras with SCP have the Vitali--Hahn--Saks property. Later on, many other completeness and interpolation properties of Boolean algebras were proved also to imply the Vitali--Hahn--Saks property, see e.g. Seever \cite{See68}, Molt\'o \cite{Mol81}, Schachermayer \cite{Sch82}, Freniche \cite{Fre84}, Aizpuru \cite{Aiz88}.

\begin{question}
For which other completeness or interpolation properties of Boolean algebras does the statement of Theorem \ref{theorem:main} hold?
\end{question}

\section{Consequences\label{section:consequences}}

In this section we provide several consequences of Theorem \ref{theorem:main} concerning cardinal characteristics of the continuum and the Efimov problem.

\subsection{Cardinal characteristics of the continuum\label{section:cardinal}}

Let us introduce the following three cardinal characteristics of the continuum.

\begin{definition}\label{def:numbers}
\textit{The Nikodym number} $\mathfrak{nik}$, \textit{the Grothendieck number} $\mathfrak{gr}$ and \textit{the Vitali--Hahn--Saks number} $\mathfrak{vhs}$ are defined respectively as:
\[\mathfrak{nik}=\min\big\{|\A|\colon\ \A\textit{ is an infinite B. algebra with the Nikodym property}\big\},\]
\[\mathfrak{gr}=\min\big\{|\A|\colon\ \A\textit{ is an infinite B. algebra with the Grothendieck property}\big\},\]
\[\mathfrak{vhs}=\min\big\{|\A|\colon\ \A\textit{ is an infinite B. a. with the Vitali--Hahn--Saks property}\big\}.\]
\end{definition}

%
%

Since every countable Boolean algebra has neither the Nikodym property nor the Grothendieck property and $\wp(\w)$ has both of the properties, we immediately get that $\w_1\le\mathfrak{nik},\mathfrak{gr}\le\mathfrak{vhs}\le\cc$. The relations between $\mathfrak{nik}$ and other classical cardinal characteristics of the continuum were studied in Sobota \cite{Sob17}, where the following inequalities were proved:
\begin{enumerate}
	\item $\max\big(\mathfrak{b},\mathfrak{s},\cov(\mathcal{M})\big)\le\mathfrak{nik}$ \cite[Corollary 3.3]{Sob17};
	\item $\w_1\le\mathfrak{nik}\le\kappa$ for every cardinal number $\kappa$ such that $\cof(\mathcal{N})\le\kappa=\cof\big([\kappa]^\omega,\subseteq\big)$ \cite[Theorem 7.3]{Sob17};
	\item $\w<\cf(\mathfrak{nik})$ and $\mathfrak{nik}$ may be consistently singular \cite[Corollary 3.7]{Sob17}.
\end{enumerate}
Results similar to (1) and (3) were obtained in Sobota \cite[Chapter 7]{Sob16} for $\mathfrak{gr}$:
\begin{enumerate}\setcounter{enumi}{3}
	\item $\max\big(\mathfrak{s},\cov(\mathcal{M})\big)\le\mathfrak{gr}$ \cite[Corollary 7.2.4]{Sob16};
	\item $\w<\cf(\mathfrak{gr})$ and $\mathfrak{gr}$ may be consistently singular \cite[Corollary 7.2.8]{Sob16}.
\end{enumerate}
Any ZFC upper bound better than $\cc$ for $\mathfrak{gr}$ has been so far unknown. However, it follows from Brech's result that in the Sacks model we have $\w_1=\mathfrak{gr}<\cc$ and hence, by Sobota and Zdomskyy \cite{SobZdo17}, $\w_1=\mathfrak{vhs}<\cc$. (Besides, note that the Grothendieck property is strongly related to the pseudo-intersection number $\mathfrak{p}$ --- see e.g. Haydon, Levy and Odell \cite[Corollary 3F]{HLO87}, Talagrand \cite{Tal80}, and Krupski and Plebanek \cite[page 2189]{KP11}.)

Theorem \ref{theorem:main} gives new situations where the numbers from Definition \ref{def:numbers} are small.

\begin{corollary}\label{cor:numbers}
Let $\IP\in V$ be a notion of forcing as in Theorem \ref{theorem:main} and let $G$ be a $\IP$-generic filter over $V$. Assume that the Continuum Hypothesis holds in $V$ but not in $V[G]$. Then, in $V[G]$, it holds $\w_1=\mathfrak{nik}=\mathfrak{gr}=\mathfrak{vhs}<\cc$.
\end{corollary}

Inequalities (1) and (4) may suggest that the dominating number $\mathfrak{d}$ is a good candidate for bounding $\mathfrak{nik}$ and $\mathfrak{gr}$ from below (e.g. Sobota \cite[Question 3.5]{Sob17}). However, using the countable support iteration of length $\w_2$ of Miller's forcing, we obtain the model where $\w_1<\mathfrak{d}=\w_2=\cc$ (see Blass \cite[Section 11.9]{Bla10}) and so, by Corollary \ref{cor:numbers}, the following holds.

\begin{corollary}\label{cor:d_big}
It is consistent that $\w_1=\mathfrak{nik}=\mathfrak{gr}=\mathfrak{vhs}<\mathfrak{d}=\w_2=\cc$.
\end{corollary}

\noindent However, we do not know whether $\mathfrak{d}$ may be consistently strictly smaller than any of the number $\mathfrak{nik}$, $\mathfrak{gr}$ or $\mathfrak{vhs}$.

\begin{question}
Let $\mathfrak{x}\in\big\{\mathfrak{nik},\mathfrak{gr},\mathfrak{vhs}\big\}$. Is it consistent that $\mathfrak{x}>\mathfrak{d}$?

In particular, is there an $\omega^\omega$-bounding poset $\IP$ such that $\wp(\omega)^V$ does not have the Vitali--Hahn--Saks property in some $\IP$-generic extension $V[G]$?
\end{question}

We can obtain a result similar to Corollary \ref{cor:d_big} for the ultrafilter number $\mathfrak{u}$ and the reaping number $\mathfrak{r}$: using the countable support iteration of length $\w_2$ of Silver's forcing, we obtain the model where $\w_1=\mathfrak{d}<\mathfrak{r}=\mathfrak{u}=\w_2=\cc$ (see Halbeisen \cite[page 379]{Hal12}).

\begin{corollary}\label{cor:r_u_big}
It is consistent that $\w_1=\mathfrak{nik}=\mathfrak{gr}=\mathfrak{vhs}<\mathfrak{r}=\mathfrak{u}=\w_2=\cc$.
\end{corollary}

It can be shown that $\w_1=\mathfrak{r}=\mathfrak{u}<\mathfrak{s}=\w_2=\cc$ consistently holds (see Blass \cite[Section 11.11]{Bla10}), so, by the inequalities (1) and (4) above and Corollary \ref{cor:r_u_big}, there is no ZFC inequality between any of the numbers from Definition \ref{def:numbers} and $\mathfrak{r}$ or $\mathfrak{u}$. A similar situation occurs also for the groupwise density number $\mathfrak{g}$: in the Cohen model we have $\w_1=\mathfrak{g}<\w_2=\cov(\mathcal{M})=\mathfrak{nik}=\mathfrak{gr}=\mathfrak{vhs}=\cc$, while in the Miller model it holds that $\w_1=\mathfrak{nik}=\mathfrak{gr}=\mathfrak{vhs}<\w_2=\mathfrak{g}=\cc$ (see Blass \cite[Chapter 11]{Bla10}).

\begin{corollary}
Let $\mathfrak{x}\in\big\{\mathfrak{nik},\mathfrak{gr},\mathfrak{vhs}\big\}$ and $\mathfrak{y}\in\big\{\mathfrak{r},\mathfrak{u},\mathfrak{g}\big\}$. Then, there is no ZFC inequality between $\mathfrak{x}$ and $\mathfrak{y}$.
\end{corollary}

Note that by (2) it follows that there is also no ZFC inequality between $\mathfrak{nik}$ and the almost disjointness number $\mathfrak{a}$. Indeed, in the Cohen model we have $\w_1=\mathfrak{a}<\cov(\mathcal{M})=\mathfrak{nik}=\omega_2=\cc$, while Brendle \cite[Proposition 4.7]{Bre03} showed that it consistently holds $\w_2=\cof(\mathcal{N})<\mathfrak{a}=\w_3=\cc$, hence consistently $\mathfrak{nik}<\mathfrak{a}$.

\begin{question}
Is it consistent that $\mathfrak{gr}<\mathfrak{a}$?
\end{question}

To obtain counterparts of Corollaries \ref{cor:d_big} and \ref{cor:r_u_big} for other cardinal characteristics (e.g. those from the right-hand side half of Cicho\'n's diagram), it would be sufficient to answer the following question.

\begin{question}
Which standard cardinal characteristics of the continuum may be pushed up to $\cc$ using a proper forcing $\IP$ having the Laver property and preserving the ground model reals non-meager?
\end{question}

We do not know whether $\mathfrak{b}\le\mathfrak{gr}$ in ZFC. If the answer for Question \ref{question:laver} is positive, then it would hold $\w_1=\mathfrak{gr}<\mathfrak{b}=\mathfrak{nik}=\mathfrak{vhs}=\w_2=\cc$ in the Laver model obtained from $V$ satisfying the Continuum Hypothesis.

\begin{question}\label{question:gr_b}
Is it consistent that $\mathfrak{gr}<\mathfrak{b}$?
\end{question}

The positive answer to Question \ref{question:gr_b} (or \ref{question:laver}) would imply that it is consistent that $\mathfrak{gr}<\mathfrak{nik}$. So far, we do not know any examples of models where the two numbers are different.

\begin{question}
Is it consistent that $\mathfrak{gr}\neq\mathfrak{nik}$?
\end{question}

\subsection{Efimov spaces\label{section:efimov}}

As we have already mentioned in Introduction, the Efimov problem is a long-standing open question asking whether there exists \textit{an Efimov space}, i.e. an infinite compact Hausdorff space with neither non-trivial converging sequences nor a copy of $\beta\w$, the \v{C}ech-Stone compactification of $\w$. Many consistent examples of Efimov spaces have been obtained, but so far no ZFC example has been found. The first consistent examples were found by Fedorchuk \cite{Fed75,Fed76} under the assumptions of the Continuum Hypothesis or $\Diamond$. Fedorchuk \cite{Fed77} also obtained an Efimov space assuming that $\mathfrak{s}=\w_1$ and $2^{\mathfrak{s}}<2^{\cc}$. Dow \cite{Dow05} strengthened Fedorchuk's result and constructed an Efimov space assuming ``only'' that $\cof\big([\mathfrak{s}]^\w,\subseteq\big)=\mathfrak{s}$ and $2^{\mathfrak{s}}<2^{\cc}$. Dow and Fremlin \cite{DF07} proved that in the random model we may have $2^{\w_1}=2^{\mathfrak{s}}=2^{\cc}$ but there still do exist Efimov spaces --- namely, they proved that if $K$ is a ground model (totally disconnected) compact F-space, then in any random generic extension $K$ has no non-trivial converging sequences. Recently, Dow and Shelah \cite{DS13} constructed an Efimov space under the assumption that $\mathfrak{b}=\cc$.

Boolean algebras with the Nikodym property or the Grothendieck property yield examples of Efimov spaces --- it is well-known that if a Boolean algebra $\A$ has either the Nikodym property or the Grothendieck property, then its Stone space $K_\A$ does not have any non-trivial convergent sequences, and hence if $\w<|\A|<\cc$, then $K_\A$ is an Efimov space (since $w\big(K_\A\big)=|\A|$). In Sobota \cite[Section 8.2]{Sob17}, assuming that $\cof(\mathcal{N})\le\kappa=\cof\big([\kappa]^\w,\subseteq\big)<\cc$, a Boolean algebra with the Nikodym property and of cardinality $\kappa$ was constructed, so an Efimov space of weight $\kappa$ was obtained as well. We can apply this argument here --- together with Theorem \ref{theorem:main} --- to prove the following corollary.

\begin{theorem}\label{theorem:efimov}
Let $\IP\in V$ be a proper forcing having the Laver property and preserving the ground model reals non-meager and $G$ a $\IP$-generic filter over $V$. Assume that the Continuum Hypothesis does not hold in $V[G]$. Let $\A\in V$ be $\sigma$-complete Boolean algebra of cardinality $\w_1$. Then, in $V[G]$, the Stone space $K_\A$ of the algebra $\A$ is an Efimov space of weight $\omega_1$.
\end{theorem}

Note that Theorem \ref{theorem:efimov} introduces a new situation for which none of the previous arguments works (e.g. those of Fedorchuk or Dow et al.), but still there does exist an Efimov space. Indeed, assume that the Generalized Continuum Hypothesis holds in the universe $V$. Let $V'$ be an extension of $V$ in which $2^{\w}=\w_1$ and $2^{\w_1}=2^{\w_2}=\w_3$ hold (e.g. force with $Fn\big(\w_3\times\w_1,2,\w_1\big)$). Then, by using the countable support iteration of length $\w_2$ of Miller's forcing obtain the extension $V''$ of $V'$ in which $\mathfrak{s}=\mathfrak{b}=\w_1<\w_2=\cof(\mathcal{N})=\cc$ and $2^{\w_1}=2^{\mathfrak{s}}=2^{\w_2}=\w_3$.

\begin{corollary}\label{corollary:new_efimov}
It is consistent that $\mathfrak{s}=\mathfrak{b}=\w_1<\w_2=\cof(\mathcal{N})=\cc$, $2^{\w_1}=2^{\mathfrak{s}}=2^{\w_2}=\w_3$, and an Efimov space exists.
\end{corollary}

There is an interesting question connecting Theorem \ref{theorem:main} and the result of Dow and Fremlin similar to Question \ref{question:laver}. Let $K$ be a totally disconnected compact space. Seever \cite[Theorem A]{See68} proved that $K$ is an F-space if and only if the Boolean algebra $\A=Clopen(K)$ of clopen subsets of $K$ has \textit{the property (I)}, i.e. for every sequences $\la A_n\colon\ n\in\w\ra$ and $\la B_n\colon\ n\in\w\ra$ in $\A$ such that $A_n\le B_m$ for every $n,m\in\w$, there exists $C\in\A$ such that $A_n\le C\le B_m$ for every $n,m\in\w$. Trivially, $\sigma$-complete Boolean algebras have the property (I). Seever \cite[Theorems B and C]{See68} proved also that if a Boolean algebra $\A$ has the property (I), then it has the Vitali--Hahn--Saks property. Now, since the Stone space of a Boolean algebra with the Nikodym property or the Grothendieck property has no non-trivial convergent sequences, we may ask about the extension of Dow and Fremlin's result as follows.

\begin{question}
Let $\IP\in V$ be the random forcing. Let $\A\in V$ be a Boolean algebra which is $\sigma$-complete (or weaker: has the property (I)). If $G$ is a $\IP$-generic filter over $V$, then does $\A$ have the Vitali--Hahn--Saks property in $V[G]$?
\end{question}

We finish the paper with the following issue concerning spaces without non-trivial converging sequences and cardinal characteristics of the continuum.

\begin{definition}
Let $\mathcal{NC}$ denotes the class of all infinite compact Hausdorff spaces without non-trivial convergent sequences. \textit{The non-convergence number} $\mathfrak{z}$ is defined as:
\[\mathfrak{z}=\min\big\{w(K)\colon\ K\in\mathcal{NC}\big\}.\]
\end{definition}

It is well-known that $\max(\mathfrak{s},\cov(\mathcal{M}))\le\mathfrak{z}$ (cf. Sobota \cite[Proposition 3.1]{Sob17}) and, by the fact that the Stone spaces of Boolean algebras with any of the Nikodym and Grothendieck properties do not contain any non-trivial convergent sequences, we have $\mathfrak{z}\le\min(\mathfrak{nik},\mathfrak{gr})$. However, we do not know whether the inequality may be strict.

\begin{question}
Is any of the inequalities $\mathfrak{z}<\mathfrak{gr}$ and $\mathfrak{z}<\mathfrak{nik}$ consistent?
\end{question}

Obviously, a positive answer to Question \ref{question:gr_b} would imply the consistency of the inequalities $\mathfrak{z}\le\mathfrak{gr}<\mathfrak{b}\le\mathfrak{nik}$. On the other hand, if the answer to Question \ref{question:gr_b} is negative, we ask about the relation between $\mathfrak{z}$ and $\mathfrak{b}$.

\begin{question}
Does $\mathfrak{b}\le\mathfrak{z}$ hold in ZFC? 
\end{question}

Note that assuming $\mathfrak{b}=\mathfrak{c}$ Dow and Shelah \cite{DS13} obtained a counterexample to the Efimov problem.

\end{document}